\documentclass{amsart}
\usepackage{amsmath, amssymb}
  \usepackage{paralist}
  \usepackage{graphics} 
  \usepackage{epsfig} 
\usepackage{graphicx}  \usepackage{epstopdf}
\usepackage[colorlinks=true]{hyperref}
\hypersetup{urlcolor=blue, citecolor=red}
\newtheorem{theorem}{Theorem}[section]
\newtheorem{corollary}[theorem]{Corollary}

\newtheorem{lemma}[theorem]{Lemma}
\newtheorem{proposition}[theorem]{Proposition}

\theoremstyle{definition}

\newtheorem{remark}[theorem]{Remark}
\newtheorem{assumption}[theorem]{Assumption}

\newtheorem{example}[theorem]{Example}
\newcommand{\ep}{\varepsilon}

\newcommand{\R}{\mathbb{R}}
\newcommand{\IC}{\mathbb{C}}
\newcommand{\IN}{\mathbb{N}}

\renewcommand{\L}{\operatorname{L}}
\newcommand{\W}{\operatorname{W}}
\newcommand{\C}{\operatorname{C}}
\newcommand{\BV}{\operatorname{BV}}
\renewcommand{\d}{\: \mathrm{d}}
\newcommand{\cl}[1]{\overline{#1}}
\DeclareMathOperator{\supp}{supp}
\newcommand{\dist}{\operatorname{dist}}
\newcommand{\gp}{C_{1,p}}
\newcommand{\cH}{\mathcal{H}}
\newcommand{\fu}{\mathfrak{u}}
\newcommand{\fv}{\mathfrak{v}}
\newcommand{\fw}{\mathfrak{w}}
\newcommand{\E}{\mathfrak{E}}
\def\Xint#1{\mathchoice
{\XXint\displaystyle\textstyle{#1}}
{\XXint\textstyle\scriptstyle{#1}}
{\XXint\scriptstyle\scriptscriptstyle{#1}}
{\XXint\scriptscriptstyle
\scriptscriptstyle{#1}}
\!\int}
\def\XXint#1#2#3{{\setbox0=\hbox{$#1{#2#3}{
\int}$ }
\vcenter{\hbox{$#2#3$ }}\kern-.6\wd0}}
\def\barint{\,\Xint -}

\title[Sobolev functions vanishing on a boundary part]
      {Characterizations of Sobolev functions that vanish on a part of the boundary}
\author{Moritz Egert}
\address{Laboratoire de Math\'{e}matiques d'Orsay \\ Univ.~Paris-Sud, CNRS, Universit\'{e} Paris-Saclay \\ 91405 Orsay, France}
\email{moritz.egert@math.u-psud.fr}

\author{Patrick Tolksdorf}
\address{Fachbereich Mathematik \\ TU Darmstadt \\ Schlossgartenstr.~7, 64289 Darmstadt, Germany}
\email{tolksdorf@mathematik.tu-darmstadt.de}


\thanks{The first author was supported by a public grant as part of the FMJH. The second author was supported by ``Studienstiftung des deutschen Volkes''}

\begin{document}

\begin{abstract}
Let $\Omega$ be a bounded domain in $\R^n$ with a Sobolev extension property around the complement of a closed part $D$ of its boundary. We prove that a function $u \in \W^{1,p}(\Omega)$ vanishes on $D$ in the sense of an interior trace if and only if it can be approximated within $\W^{1,p}(\Omega)$ by smooth functions with support away from $D$. We also review several other equivalent characterizations, so to draw a rather complete picture of these Sobolev functions vanishing on a part of the boundary.
\end{abstract}
\maketitle
\section{Introduction}
\label{Sec:Introduction}

\noindent In this note we study first-order Sobolev spaces on a bounded domain $\Omega \subseteq \R^n$, $n~\geq~2$, encapsulating a Dirichlet boundary condition on a closed part $D$ of the boundary $\partial \Omega$. These function spaces appear quite naturally in the variational treatment of elliptic and parabolic divergence-form problems if the solution should satisfy a Dirichlet condition only on one part of the boundary, whereas on the complementary part other restrictions are imposed. For a comprehensive treatment and specific, physically relevant examples of such mixed boundary value problems the reader can refer to~\cite{RobertJDE}.

In these applications the underlying domain typically is too rough as to admit a trace operator for the whole Sobolev space $\W^{1,p}(\Omega)$, defined as the collection of all $u \in \L^p(\Omega)$ such that in the sense of distributions $\nabla u \in \L^p(\Omega)^n$. On the other hand, classical regularity results for solutions of mixed boundary value problems such as H\"older continuity are still available, see the recent developments in~\cite{Rehberg-terElstADE} and references therein. This motivates to investigate in which sense the Dirichlet boundary condition `$u = 0$ on $D$' can be understood if only $u \in \W^{1,p}(\Omega)$ holds.

Particularly with regard to mixed boundary value problems, the weakest meaningful definition of a closed subspace of $\W^{1,p}(\Omega)$, $1 < p < \infty$, incorporating the Dirichlet boundary condition on $D$ is given by approximation: The space $\W^{1,p}_D(\Omega)$ is defined as the closure in $\W^{1,p}(\Omega)$ of the set of test functions
\begin{align*}
 \C_D^\infty(\Omega) := \{u|_\Omega : u \in \C_0^\infty(\R^n),\, \supp(u) \cap D = \emptyset \}
\end{align*}
with support away from the closed set $D$. More generally, this definition makes sense if only $\Omega$ is open and $D \subseteq \cl{\Omega}$ is closed. Just recently, the structure of the spaces $\W^{1,p}_D(\Omega)$ has been studied with the objective of obtaining intrinsic characterizations that only use the given Sobolev function $u \in \W^{1,p}(\Omega)$ in order to decide whether or not $u \in \W^{1,p}_D(\Omega)$ holds. On the whole space $\Omega = \R^n$ this problem is perfectly understood due to the Havin-Bagby-Theorem \cite[Thm.~9.1.3]{Adams-Hedberg}, making use of the notion of \emph{$p$-capacity} of sets $E \subseteq \R^n$,
\begin{align*}
 \gp(E) := \inf \bigg\{ \int_{\R^n} |f|^p \d y : f \geq 0 \text{ on $\R^n$}, \, G_1 \ast f \geq 1 \text{ on $E$}\bigg\},
\end{align*}
where $1<p<\infty$ and $G_1 \in \L^1(\R^n)$ is the first-order Bessel kernel defined as the inverse Fourier transform of $\xi \mapsto (1+|\xi|^2)^{-1/2}$.

\begin{proposition}[The Havin-Bagby-Theorem]
\label{prop:HavinBagby}
Let $1<p<\infty$, let $E \subseteq \R^n$ be closed, and let $v \in \W^{1,p}(\R^n)$. Then $v \in \W^{1,p}_E(\R^n)$ if and only if for $\gp$-almost every $x \in E$,
\begin{align}
\label{eq:HavinBagby}
 \lim_{r \to 0} \frac{1}{|B(x,r)|} \int_{B(x,r)} v \d y = 0.
\end{align}
\end{proposition}

As for domains, it was observed for instance in~\cite{TripleMitrea-Brewster} and \cite{Hardy-Poincare} that under suitable geometric assumptions every $u \in \W^{1,p}_D(\Omega)$ can be extended to a function $v$ not only in $\W^{1,p}(\R^n)$ but in $\W^{1,p}_D(\R^n)$. Consequently, the Havin-Bagby-Theorem can also be used to describe $\W^{1,p}_D(\Omega)$. While this characterization is `intrinsic' \cite{TripleMitrea-Brewster} in that it does not depend on the particular choice of the extension, it is certainly not canonical as it is given in terms of a Sobolev function different from $u$ somehow to be chosen yet. For a Sobolev function defined only on $\Omega$ the natural substitute for \eqref{eq:HavinBagby} is to require 
\begin{align}
\label{eq:InteriorTrace}
 \lim_{r \to 0} \frac{1}{|B(x,r)|} \int_{B(x,r) \cap \Omega} |u| \d y = 0 \qquad (\text{for $\gp$-almost every $x \in D$})
\end{align}
and the purpose of this note is to prove that under the following geometric assumption this interior trace condition indeed provides a new, canonical characterization of $\W^{1,p}_D(\Omega)$.

\begin{assumption}
\label{ass:W1pExtension}
The \emph{$\W^{1,p}$-extension property} holds around $\cl{\partial \Omega \setminus D}$, that is, every $x \in \cl{\partial \Omega \setminus D}$ has an open neighborhood $U_x$ such that $U_x \cap \Omega$ is connected and admits a bounded extension operator $\E_x: \W^{1,p}(U_x \cap \Omega) \to \W^{1,p}(\R^n)$.
\end{assumption}

By an extension operator we always mean a linear operator that does not modify functions on the smaller domain. Assumption~\ref{ass:W1pExtension} allows us to construct a bounded extension operator $\E: \W^{1,p}_D(\Omega) \to \W^{1,p}_D(\R^n)$ via a localization argument~\cite[Thm.~6.9]{Hardy-Poincare}, thereby making the Havin-Bagby-Theorem applicable as discussed above. As for mixed boundary value problems, this geometric assumption is rather common since it seems to be indispensable for treating most non-Dirichlet boundary conditions on $\cl{\partial \Omega \setminus D}$. Let us mention that it covers the more specific case of a bounded domain $\Omega$ exhibiting Lipschitz coordinate charts around $\cl{\partial \Omega \setminus D}$. For a further discussion the reader can refer to~\cite[Sec.~6.4]{Hardy-Poincare}. 

Somewhat hidden at first sight, one of the most important features of Assumption~\ref{ass:W1pExtension} is that it guarantees a certain regularity of $\Omega$ near the common frontier of $D$ with the complementary boundary part by requiring the $\W^{1,p}$-extension property around the \emph{closure} of $\partial \Omega \setminus D$. In fact, if the $\W^{1,p}$-extension property only holds around $\partial \Omega \setminus D$, then \eqref{eq:InteriorTrace} is neither necessary nor sufficient for $u \in \W^{1,p}(\Omega)$ to be a member of $\W^{1,p}_D(\Omega)$, see Section~\ref{Sec:Examples} for explicit counterexamples.

Assumption~\ref{ass:W1pExtension} is void if pure Dirichlet conditions $D = \partial \Omega$ are imposed and in this case the conclusion that \eqref{eq:InteriorTrace} characterizes $\W^{1,p}_D(\Omega) = \W^{1,p}_0(\Omega)$ is due to Swanson and Ziemer~\cite{Swanson-Ziemer}. We shall review their proof in Section~\ref{Sec:Ziemer-Swanson}, not only for convenience but also since our approach requires all details of theirs.   

The integrals in \eqref{eq:InteriorTrace} can be replaced by true averages if $\Omega$ satisfies the lower density condition
\begin{align*}
 \liminf_{r \to 0} \frac{|B(x,r) \cap \Omega|}{|B(x,r)|} > 0
\end{align*}
around $C_{1,p}$-every $x \in D$. However, we stress that this need not be the case, neither in the context of mixed boundary value problems nor within the setup of this note. Note also that \eqref{eq:InteriorTrace} -- in contrast to \eqref{eq:HavinBagby} -- uses the absolute value of $u$. This modification is necessary since our geometric assumptions do not rule out that the boundary part $D$ is contained in the interior of the closure of $\Omega$. In particular, we may think of a rectangle $R := (-2,2) \times (-4,4)$ in $\R^2$ sliced by $D := \{0\} \times (-2,2)$, and define the bounded domain $\Omega := R \setminus D$. Then any $v \in \W^{1,p}(\Omega)$ that takes the constant values $-1$ and $1$ on the subregions $(-1,0) \times (-1,1)$ and $(0,1) \times (-1,1)$, respectively, will satisfy \eqref{eq:HavinBagby} everywhere on $\{0\} \times (-1,1)$, which for any choice of $p$ is a set of positive $p$-capacity in the plane~\cite[Thm.~2.6.16]{Ziemer}.

Let us close by remarking that in the context of mixed boundary value problems the Dirichlet part $D$ typically is not just closed but satisfies for some $l \in (0,n]$ an additional density assumption with respect to the $l$-dimensional Hausdorff measure $\cH_l$ on $\R^n$,
\begin{align}
\label{eq:Ahlfors}
 \cH_l(B(x,r) \cap D) \sim r^l \qquad (x \in D,\, r< 1),
\end{align}
which is usually referred to as \emph{$l$-Ahlfors regularity}. In this case the capacities entering in \eqref{eq:HavinBagby} and \eqref{eq:InteriorTrace} can often be replaced with coarser and easier to handle Hausdorff measures. Moreover, for such geometric configurations there is yet another intrinsic characterization of $\W^{1,p}_D(\Omega)$ of a rather different nature: It relies on Hardy's inequality, that is, integration against the weight $x \mapsto \dist_D(x)^{-p}$, which is singular at the Dirichlet part~\cite[Thm.~3.2 \& 3.4]{Hardy-Poincare}. Here, $\dist_D$ denotes the Euclidean distance function to the closed set $D$.
\section{The main result}
\label{Sec:Result}

\noindent Besides the alluded interior trace result, we also see this note as good opportunity to concisely list the so-far known equivalent conditions for a function in $\W^{1,p}(\Omega)$ to vanish on $D$ in the weakest possible sense. This is being done in our following main theorem.

\begin{theorem}
\label{thm:main}
Let $\Omega \subseteq \R^n$ be a bounded domain, let $D$ be a closed subset of its boundary, and let $1<p<\infty$. Under Assumption~\ref{ass:W1pExtension} the following are equivalent for any given $u \in \W^{1,p}(\Omega)$.
\begin{enumerate}
 \item[(i)] The function $u$ belongs to $\W^{1,p}_D(\Omega)$.
 \item[(ii)] For $\gp$-almost every $x \in D$ it holds 
	\begin{align*} \lim_{r \to 0} \frac{1}{|B(x,r)|} \int_{B(x,r) \cap \Omega} |u| \d y = 0. \end{align*}
 \item[(iii)] There exists a Sobolev extension $v \in \W^{1,p}(\R^n)$ of $u$ that satisfies for $\gp$-almost every $x \in D$,
	\begin{align*} \lim_{r \to 0} \frac{1}{|B(x,r)|} \int_{B(x,r)} v \d y = 0. \end{align*}
\end{enumerate}
If in addition $D$ is $l$-Ahlfors regular and $n-p < l < n$, then these conditions are also equivalent to the following.
\begin{enumerate}
 \item[(iv)] There exists a Sobolev extension $v \in \W^{1,p}(\R^n)$ of $u$ that satisfies for $\cH_l$-almost every $x \in D$,
	\begin{align*} \lim_{r \to 0} \frac{1}{|B(x,r)|} \int_{B(x,r)} v \d y = 0. \end{align*}      
 \item[(v)] The function $u$ satisfies the Hardy-type condition
        \begin{align*}
         \int_\Omega \bigg|\frac{u}{\dist_D}\bigg|^p \d y < \infty.
        \end{align*}
\end{enumerate}
If even $n-p < l \leq n-1$, then the conditions above are also equivalent to the following.
\begin{enumerate}
\item[(vi)] For $\cH_l$-almost every $x \in D$ it holds 
	\begin{align*} \lim_{r \to 0} \frac{1}{|B(x,r)|} \int_{B(x,r) \cap \Omega} |u| \d y = 0. \end{align*}
\end{enumerate}
\end{theorem}

\begin{remark}
\label{rem:Hardy}
If the Hardy-type condition in (v) holds true for every $u \in \W_D^{1,p}(\Omega)$, then every such $u$ also satisfies Hardy's inequality
\begin{align*}
 \int_\Omega \bigg|\frac{u}{\dist_D}\bigg|^p \d y \lesssim \int_\Omega |u|^p + |\nabla u|^p \d y.
\end{align*}
Indeed, this is a consequence of the closed graph theorem applied to the multiplication operator with symbol $\dist_D(x)^{-p}$.
\end{remark}

\begin{remark}
\label{rem:condition6}
Even though the restriction $l \leq n-1$ in (vi) compared to $l < n$ in (iv) is of no harm for applications to mixed boundary value problems (where all too often $l = n-1$), the question whether it is needed as part of our main result remains open. It will become clear in Section~\ref{Sec:Proof} that the answer in the affirmative would require a rather different argument.
\end{remark}

Since first-order Sobolev spaces are invariant under truncation, $|u| \in \W^{1,p}(\Omega)$ holds for every $u \in \W^{1,p}(\Omega)$. The equivalence of (i) and (ii) in Theorem~\ref{thm:main} implies the following worth mentioning corollary.

\begin{corollary}
\label{cor:truncationD}
Presume the setup of Theorem~\ref{thm:main} and let $u \in \W^{1,p}(\Omega)$. Then $u \in \W^{1,p}_D(\Omega)$ if and only if $|u| \in \W^{1,p}_D(\Omega)$.
\end{corollary}

In Section~\ref{Sec:Proof} we shall give complete proofs of the new implications in Theorem~\ref{thm:main} and provide solid references for the already known ones. In the preliminary Section~\ref{Sec:Preliminaries} we collect some classical continuity properties of Sobolev functions and in Section~\ref{Sec:Ziemer-Swanson} we shall review Swanson and Ziemer's argument for $\W_0^{1,p}(\Omega)$ in order to set the stage for the general case.
\section{Continuity properties of Sobolev functions}
\label{Sec:Preliminaries}

\noindent A locally integrable function $f: \R^n \to \IC$ possesses a \emph{Lebesgue point} at $x \in \R^n$ if there exists a number $l = l(x)$ such that
\begin{align*}
 \lim_{r \to 0} \barint_{B(x,r)} |f(y) - l| \d y = 0.
\end{align*}
Here and throughout, averages are materialized by a dashed integral. We say that $f$ is \emph{approximately continuous} at $x \in \R^n$ if there exists a measurable set $E_x$ of full Lebesgue density at $x$, that is
\begin{align*}
 \lim_{r \to 0} \frac{|B(x,r) \cap E_x|}{|B(x,r)|} = 1,
\end{align*}
such that 
\begin{align*}
 \lim_{E_x \ni y \to x} f(y) = f(x).
\end{align*}
Lebesgue points and points of approximate continuity are related via the following lemma from classical measure theory \cite[Ch.~3,\, Sec.~1.4]{currents}.

\begin{lemma}
\label{lem:ApprCont}
Let $f: \R^n \to \IC$ be locally integrable. If $f$ possesses a Lebesgue point at $x \in \R^n$ with $l(x) = f(x)$, then $f$ is approximately continuous at $x$.
\end{lemma}

Next, let us recall that $p$-capacities and Lebesgue points for Sobolev functions $v \in \W^{1,p}(\R^n)$, $1<p<\infty$, are intrinsically tied to each other by the fact that the limit of averages
\begin{align*}
 \lim_{r \to 0} \barint_{B(x,r)} v \d y =: \fv(x)
\end{align*}
is finite for $\gp$-almost every $x \in \R^n$. The so-defined function $\fv$ reproduces $v$ within its Lebesgue class and is called \emph{precise representative} of $v$. For convenience we set $\fv(x) = 0$ if the limit above does not exist. The Lebesgue Differentiation Theorem for Sobolev functions asserts that for $\gp$-almost every $x \in \R^n$ we have
\begin{align*}
 \lim_{r \to 0} \barint_{B(x,r)} |v(y) - \fv(x)|^p \d y = 0.
\end{align*}
In particular, $\gp$-almost every point $x \in \R^n$ is a Lebesgue point for $\fv$ with $l(x) = \fv(x)$ and hence $\fv$ is approximately continuous $\gp$-almost everywhere. Often we shall not distinguish between $v$ and $\fv$ and simply speak of approximate continuity of $v$. The reader can refer to \cite[Sec.~3]{Ziemer} for proofs of these facts and further background. 

As a second continuity principle for Sobolev functions we need the following result \cite[Thm.~2.1.4]{Ziemer}. When speaking of properties that hold on \emph{almost all lines parallel to the $x_k$-axis}, where $1 \leq k \leq n$, we think of the supporting line as being identified with its base point in $\R^{n-1}$ and use the $(n-1)$-dimensional Lebesgue measure.

\begin{proposition}
\label{prop:continuitylines}
Let $U \subseteq \R^n$ be open, $1 < p < \infty$, and $u \in \L^p(U)$. Then $u \in \W^{1,p}(U)$ if and only if $u$ has a representative $\tilde{u}$ that is absolutely continuous on every compact interval contained in $U$ of almost all lines that are parallel to the coordinate axes and whose classical partial derivatives belong to $\L^p(U)$.
\end{proposition}

We also require basic knowledge on functions of bounded variation in several variables and refer to \cite[Ch.~5]{Ziemer} or \cite[Ch.~5]{EG} for further reading. The space $\BV(U)$ of functions of bounded variation on an open set $U \subseteq \R^n$ consists of all integrable functions $v$ on $U$ whose distributional partial derivatives are totally finite Radon measures on $U$. The next result, found for example in \cite[Thm.~5.3.5]{Ziemer} or \cite[Sec.~5.10.2]{EG}, provides the link with the classical one-dimensional notion of bounded variation. There, we define the \emph{essential variation} of a scalar-valued function $g$ on a closed interval $[a , b]$  by
\begin{align*}
 \mathrm{essV}_a^b (g) := \sup \Big\{ \sum_{i = 1}^k |g(t_i) - g(t_{i - 1})| \Big\},
\end{align*}
where the supremum is taken over all finite partitions of $[a , b]$ induced only by points $a < t_0 < t_1 < \dots < t_k < b$ at which $g$ is approximately continuous.

\begin{proposition}
\label{prop:BVcharOnLines}
Let $v \in \BV(\R^n)$. Fix a rectangular cell $R \subseteq \R^{n - 1}$, a space direction $1 \leq k \leq n$ of $\R^n$, and real numbers $a_k < b_k$. Denote points in $\R^n$ by $(x', x_k) \in \R^{n-1} \times \R$ and let $v_{x'} := v(x', \cdot)$ be the restriction of $v$ to the line parallel to the $x_k$-axis passing through $(x',0)$. Then
\begin{align*}
 \int_R \mathrm{essV}_{a_k}^{b_k} (v_{x'}) \d x^{\prime} < \infty.
\end{align*}
\end{proposition}

The following extension result for functions of bounded variation is due to Swanson and Ziemer~\cite[Thm.~2.1]{Swanson-Ziemer}. By the \emph{zero extension} of a function $v$ defined on a set $U \subseteq \R^n$ we mean the trivial continuation of $v$ to the whole space by $0$.

\begin{proposition}
\label{prop:abscontImpliesBV}
Let $U \subseteq \R^n$ be an open set and let $u$ be a function defined on $U$ with the property that $u \in \BV(U^{\prime})$ for every open and bounded subset $U^{\prime} \subseteq U$. If the zero extension $u^*$ of $u$ is approximately continuous at $\cH_{n - 1}$-almost every $x \in \R^n$, then $u^* \in \BV (U^{\prime})$ for every open bounded subset $U^{\prime} \subseteq \R^n$.
\end{proposition}

We close by stating two related results that will prove to be useful in the further course. Their proofs can be found in the textbooks \cite[Thm.~4.5.9(29)]{Federer} and \cite[Thm.~13.8]{Yeh}, respectively.

\begin{proposition}
\label{prop:BVandApprocContImpliesContOnLines}
If $v \in \BV (\R^n)$ is approximately continuous at $\cH_{n - 1}$-almost every $x \in \R^n$, then $v$ is continuous on almost all lines parallel to the coordinate axes. 
\end{proposition}

\begin{proposition}[Banach-Zarecki Criterion]
\label{prop:charAbsCont}
A scalar-valued function $f$ on a compact interval is absolutely continuous if and only if it is continuous, of bounded variation, and carries sets of Lebesgue measure zero into sets of Lebesgue measure zero.
\end{proposition}
\section{A review of Swanson and Ziemer's argument}
\label{Sec:Ziemer-Swanson}

\noindent In this section we review Swanson and Ziemer's \cite{Swanson-Ziemer} proof of `(ii) $\Longrightarrow$ (i)' in the case of pure Dirichlet conditions. Along the way we shall reveal a useful addendum to their result that is recorded as the second part of the following proposition. Let us stress that here the restriction to \emph{bounded} open sets is only for the sake of simplicity, compare with \cite{Swanson-Ziemer}.

\begin{proposition}
\label{Prop:SZ}
Let $U \subseteq \R^n$ be a bounded open set and let $1<p<\infty$. If $u \in \W^{1,p}(U)$ has the property
\begin{align*}
 \lim_{r \to 0} \frac{1}{|B(x,r)|} \int_{B(x,r) \cap U} |u| \d y = 0
\end{align*}
for $\gp$-almost every $x \in \partial U$, then $u \in \W^{1,p}_0(U)$. If $u$ has this property only for $\cH_l$-almost every $x \in \partial U$ and if $n-p < l \leq n-1$, then its zero extension $u^*$ is at least contained in $\W^{1,p}(\R^n)$ and satisfies
\begin{align*}
 \lim_{r \to 0} \frac{1}{|B(x,r)|} \int_{B(x,r)} |u^*| \d y = 0
\end{align*}
for $\cH_l$-almost every $x \in {}^c U$.
\end{proposition}

The following comparison principle asserts that the assumption in the second part of the proposition is indeed the weaker one. For a proof the reader can refer to \cite[Sec.~5]{Adams-Hedberg} for the case $p \leq n$ and \cite[Prop.~2.6.1]{Adams-Hedberg} for the case $p>n$.

\begin{lemma}
\label{lem:ComparisonPrinciple}
If $1 < p < \infty$, $l>0$, and $n-p< l \leq n$, then every set $E \subseteq \R^n$ of vanishing capacity $\gp(E) = 0$ also satisfies $\cH_l(E) = 0$.
\end{lemma}

\begin{proof}[Proof of Proposition~\ref{Prop:SZ}]
The argument is in six consecutive steps. As Lebesgue points and points of approximate continuity are local properties, we can associate with $u$ a precise representative $\fu$ as in Section~\ref{Sec:Preliminaries}. Then we define a representative $\fu^*$ of the zero extension $u^* \in \L^p(\R^n)$ by $\fu^*(x) := \fu(x)$ if $x \in U$ and $\fu^*(x) := 0$ if $x \in {}^c U$. 

\medskip

\emph{Step~1: $\fu^*$ is approximately continuous $\cH_{n- 1}$-almost everywhere.} Recall from Section~\ref{Sec:Preliminaries} that $\fu$ is approximately continuous at $\gp$-almost every $x \in U$, hence at $\cH_{n-1}$-almost every $x \in U$ due to Lemma~\ref{lem:ComparisonPrinciple}. Again by this lemma and since every set of vanishing $\cH_l$-measure has vanishing $\cH_{n - 1}$-measure provided $l \leq n-1$, we obtain under both conditions of the proposition that for $\cH_{n - 1}$-almost every $x \in \partial U$ it holds
\begin{align*}
 \lim_{r \to 0} \frac{1}{|B(x , r)|} \int_{B(x , r)} |\fu^* - \fu^*(x)| \d y = \lim_{r \to 0} \frac{1}{|B(x , r)|} \int_{B(x , r) \cap U} |u| \d y = 0.
\end{align*}
Thus, $\fu^*$ is approximately continuous at these boundary points owing to Lemma~\ref{lem:ApprCont}. Finally, $\fu^*$ is identically zero on the open set ${}^c \cl{U}$ and hence (approximately) continuous at every $x \in {}^c \cl{U}$.

\medskip

\emph{Step~2: $\fu^*$ if of bounded variation on $\R^n$.} We simply have to combine Proposition~\ref{prop:abscontImpliesBV} with the first step and recall that $\fu^*$ vanishes outside of a bounded set.

\medskip

\emph{Step~3: $\fu^*$ is continuous on almost all lines parallel to the coordinate axes.} This is a direct consequence of the first two steps and Proposition~\ref{prop:BVandApprocContImpliesContOnLines}.

\medskip

\emph{Step~4: $\fu^*$ is of bounded variation on every compact interval of almost all lines parallel to the coordinate axes.} Combining Step~2 with Proposition~\ref{prop:BVcharOnLines}, we obtain that the essential variation of $\fu^*$ is bounded along every compact interval of almost all lines parallel to the coordinate axes. In view of Step~3 we may additionally assume that the restriction of $\fu^*$ to the respective lines is continuous and thus approximately continuous at every point. Hence, the definition of the essential variation collapses to the one of the standard one-dimensional variation and the claim follows.

\medskip

\emph{Step~5: $\fu^*$ is absolutely continuous on every compact interval of almost all lines parallel to the coordinate axes.} Due to Proposition~\ref{prop:charAbsCont} and the outcome of Steps~3 and 4 we only have to show that on almost all lines parallel to the coordinate axes the restriction of $\fu^*$ maps sets of one-dimensional Lebesgue measure zero into sets of one-dimensional Lebesgue measure zero. 

To this end let $\lambda$ be a line parallel to the $x_k$-axis passing through the point $x = (x^{\prime} , x_k)$, where we adopt notation from Proposition~\ref{prop:BVcharOnLines}. Owing to Steps~3 and 4 we may assume that $\fu^*(x^{\prime} , \cdot)$ is continuous and of bounded variation. Proposition~\ref{prop:continuitylines} provides yet another representative $\tilde{u}$ of $u \in \W^{1,p}(U)$ that is absolutely continuous on every compact interval contained in $U$ of almost all lines parallel to the $x_k$-axis. We may assume that this applies to $\lambda$ and in view of Fubini's theorem we may as well assume that $\fu(x',\cdot)$ and $\tilde{u}(x',\cdot)$ coincide almost everywhere on $\lambda \cap U$ with respect to the one-dimensional Lebesgue measure. By continuity they have to agree everywhere on $\lambda \cap U$, showing that we may additionally assume that $\fu^*(x',\cdot)$ is absolutely continuous on every compact interval contained in $\lambda \cap U$. 

Let now $E \subseteq \lambda$ be a set of vanishing one-dimensional Lebesgue measure. Being the zero extension of $\fu$, the function $\fu^*$ maps $E \cap {}^c U$ onto $\{0\}$, so that it remains to investigate what happens to the set $E \cap U$. To this end, let $I$ be an open subinterval of $\lambda \cap U$ and let $J \subseteq I$ be a compact interval. Then Proposition~\ref{prop:charAbsCont} guarantees that $\fu^* (E \cap J)$ has vanishing one-dimensional Lebesgue measure. In virtue of the regularity of the Lebesgue measure this property first carries over to $\fu^* (E \cap I)$ and then to $\fu^*(E \cap U)$.

\medskip

\emph{Step~6: Conclusion of the proof.} Due to Step~5 the classical partial derivatives of $\fu^*$ exist almost everywhere on almost all lines parallel to the coordinate axes. Since the restriction of $\fu^*$ to $U$ is a representative for $u \in \W^{1,p}(U)$,  Proposition~\ref{prop:continuitylines} yields that the classical partial derivatives of $\fu^*$ evaluated at points inside $U$ define $p$-integrable functions on $U$. Since $\fu^*$ vanishes on the open set ${}^c \cl{U}$, so do its classical partial derivatives. It remains to investigate the critical case, that is, the behavior at the boundary of $U$.

To this end, let $\lambda$ be one of the lines parallel to the coordinate axes on which $\fu^*$ has the differentiability properties above. Let $x \in \lambda \cap \partial U$ be such that the classical partial derivative of $\fu^*$ in the direction of $\lambda$ exists at $x$. 

By a topological case distinction, either there exists an open one-dimensional neighborhood $I \subseteq \lambda$ of $x$ such that $I \cap {}^c U = \{ x \}$ or $x$ can be approximated by a sequence of points $(x_j)_{j \in \IN} \subseteq \lambda \cap {}^c U$ that are all distinct from $x$. In the second case the classical partial derivative of $\fu^*$ at $x$ in the direction of $\lambda$ vanishes since $\fu^*(x) = \fu^*(x_j) = 0$ for all $j$. The first case looks rather odd but anyway it can occur at most countably often on $\lambda$ since $I$ is open and $x$ is the only point in $I$ with this property. Thus, without even investigating this first case, we can conclude that the classical partial derivative of $\fu^*$ in direction of $\lambda$ vanishes at almost every point of $\lambda \cap \partial U$ with respect to the one-dimensional Lebesgue measure. 

Taking into account Fubini's theorem, we can conclude that the classical partial derivatives of $\fu^*$ are $p$-integrable over $\R^n$. Consequently, Proposition~\ref{prop:continuitylines} yields that $u^*$ is contained in $\W^{1, p} (\R^n)$ and this already concludes the proof of the second statement of the proposition. In the first case we may now apply the Havin-Bagby-Theorem to $u^* \in \W^{1 , p}(\R^n)$ and obtain $u^* \in \W^{1 , p}_{{}^c U}(\R^n)$ from the assumption on $u$ and the fact that $u^*$ vanishes on $^{c} \cl{U}$. This precisely means $u \in \W^{1 , p}_0 (U)$.
\end{proof}
\section{Proof of the main result}
\label{Sec:Proof}

\noindent The proof of Theorem~\ref{thm:main} will be achieved through the eight implications below.

\subsection*{(iii) \texorpdfstring{$\Longrightarrow$}{=>} (i)}
If $u \in \W^{1,p}(\Omega)$ has a Sobolev extension $v \in \W^{1,p}(\R^n)$ that satisfies
\begin{align*} 
 \lim_{r \to 0} \frac{1}{|B(x,r)|} \int_{B(x,r)} v \d y = 0 
\end{align*}    
for $\gp$-almost every $x \in D$, then $v \in \W^{1,p}_D(\R^n)$ thanks to the Havin-Bagby-Theorem. By definition, this means that $v$ is contained in the $\W^{1,p}(\R^n)$-closure of $\C_D^\infty(\R^n)$. Restricting to $\Omega$, we find that $v|_\Omega$ is contained in the $\W^{1,p}(\Omega)$-closure of $\C_D^\infty(\Omega)$ and due to $v|_\Omega = u$ the conclusion $u \in \W^{1,p}_D(\Omega)$ follows.

\subsection*{(ii) \texorpdfstring{$\Longrightarrow$}{=>} (i)}
This is of course the most interesting implication. A part of the argument was inspired by \cite[Sec.~VIII.1]{Jonsson-Wallin}. Let us set some notation for a localization argument first. For $x \in \cl{\partial \Omega \setminus D}$ we let $U_x$ be as in Assumption~\ref{ass:W1pExtension} and pick a finite subcovering $U_{x_1}, \ldots, U_{x_N}$ of the compact set $\overline{\partial \Omega \setminus D}$. Then there exists $\ep > 0$ such that $U_{x_1},\ldots,U_{x_N}$ together with
\begin{align*}
 U_0 := \big\{y \in \R^n : \dist(y, \overline{\partial \Omega \setminus D}) > \ep \big\}
\end{align*}
form an open covering of $\overline{\Omega}$. Thus, on $\cl{\Omega}$ there is a $\C^\infty$-partition of unity $\eta_0, \ldots, \eta_N$ with the properties $0 \leq \eta_j \leq 1$ and $\supp(\eta_j) \subseteq U_j$. Here and in the following we abbreviate $U_{x_j}$ by $U_j$.

Now assume that $u \in \W^{1,p}(\Omega)$ satisfies (ii). We split $u = \sum_{j=0}^N u_j$, where the functions $u_j := \eta_j u$ are all contained in $\W^{1,p}(\Omega)$. We shall prove that each summand is in fact contained in $\W^{1,p}_D(\Omega)$. 

\medskip

\emph{Step 1: The case $j = 0$}. By assumption on $u$ we have for $\gp$-almost every $x \in D$ the limits
\begin{align*}
  0 \leq \lim_{r \to 0} \frac{1}{|B(x,r)|} \int_{B(x,r) \cap \Omega} |u_0| \d y \leq \lim_{r \to 0} \frac{1}{|B(x,r)|} \int_{B(x,r) \cap \Omega} |u| \d y = 0
\end{align*}
and for every $x \in \partial \Omega \setminus D$ the choice of $\eta_0$ implies
\begin{align*}
 \frac{1}{|B(x,r)|} \int_{B(x,r) \cap \Omega} |u_0| = 0
\end{align*}
provided $r \leq \ep$. Hence, for $\gp$-almost every $x \in \partial \Omega$ we have
\begin{align*}
 \lim_{r \to 0} \frac{1}{|B(x,r)|} \int_{B(x,r) \cap \Omega} |u_0| \d y = 0
\end{align*}
and Proposition~\ref{Prop:SZ} yields $u_0 \in \W^{1,p}_0(\Omega) \subseteq \W^{1,p}_D(\Omega)$ as required.

\medskip

\emph{Step 2: Preliminaries for the case $j \geq 1$}. Consider a summand $u_j$ with $j \geq 1$. Assumption~\ref{ass:W1pExtension} allows us to construct an extension $v_j \in \W^{1,p}(\R^n)$ that coincides with $u_j$ on the domain $U_j \cap \Omega$. We can further assume that $v_j$ is supported in $U_j$ and agrees with $u_j$ almost everywhere on $\Omega$ since otherwise we would replace $v_j$ by the extension $\chi v_j$, where $\chi \in \C_0^\infty(U_j)$ is identically $1$ on the support of $\eta_j$. Since $\W^{1,p}(\R^n)$ is invariant under truncation, we also have $w := |v_j|$ in this space. Let now $x \in D$ be such that the limits
\begin{align*}
 \fw(x) = \lim_{r \to 0} \frac{1}{|B(x,r)|} \int_{B(x,r)} w \d y \quad \text{and} \quad \lim_{r \to 0} \frac{1}{|B(x,r)|} \int_{B(x,r)} |w - \fw(x)| \d y
\end{align*}
exist and such that additionally
\begin{align}
\label{eq0:iitoi}
 \lim_{r \to 0} \frac{1}{|B(x,r)|} \int_{B(x,r) \cap \Omega} |u| \d y = 0
\end{align}
holds. Here, $\fw$ denotes the precise representative of $w$. By assumption and the Lebesgue Differentiation Theorem from Section~\ref{Sec:Preliminaries} the three conditions can simultaneously be matched for $\gp$-almost every $x \in D$. For the moment our task is to demonstrate $\fw(x) = 0$. In doing so, only the case $x \in U_j$ is of interest since $w = |v_j|$ has support in $U_j$.

\medskip

\emph{Step 3: Proof of $\fw(x) = 0$}. Lemma~\ref{lem:ApprCont} provides a measurable set $E_x$ with
\begin{align}
\label{eq1:iitoi}
 \lim_{r \to 0} \frac{|B(x,r) \cap E_x|}{|B(x,r)|} = 1
\end{align}
and the property that the restriction $\fw|_{E_x}$ is continuous at $x$. In order to simplify notation, we shall abbreviate $E_x = E$ and $B(x,r) = B(r)$ in the following. For $r>0$ we may write
\begin{align*}
 \frac{1}{|B(r)|} \int_{B(r)} w \d y = \frac{|B(r) \cap E|}{|B(r)|} \barint_{B(r) \cap E} \fw \d y + \frac{1}{|B(r)|} \int_{B(r) \cap {}^cE} w \d y.
\end{align*}
Here, the left-hand term tends to $\fw(x)$ in the limit $r \to 0$ by assumption on $x$ and so does the first term on the right-hand side thanks to \eqref{eq1:iitoi} and the continuity of $\fw|_{E}$ at $x$. Hence, the conclusion
\begin{align}
\label{eq2:iitoi}
\lim_{r \to 0} \frac{1}{|B(r)|} \int_{B(r) \cap {}^c E} w \d y = 0.
\end{align}
Since $w = |v_j| = |u_j|$ almost everywhere on $\Omega$, we can also decompose
\begin{align*}
 \frac{1}{|B(r)|} \int_{B(r)} w \d y 
 &= \frac{1}{|B(r)|} \int_{B(r) \cap \Omega} |u_j| \d y + \frac{1}{|B(r)|} \int_{B(r) \cap {}^c \Omega \cap E} \fw \d y \\
 &\quad+ \frac{1}{|B(r)|} \int_{B(r) \cap {}^c \Omega \cap {}^cE} w \d y.
\end{align*}
Again we investigate the behavior in the limit $r \to 0$: The left-hand term tends to $\fw(x)$ as before. From the pointwise bound $|u_j| \leq |u|$ and \eqref{eq0:iitoi} we deduce that the first term on the right-hand side vanishes. For the third term we obtain the same conclusion, this time using \eqref{eq2:iitoi}  and that $w$ is a nonnegative function. Altogether, we have found
\begin{align}
\label{eq3:iitoi}
\fw(x) = \lim_{r \to 0} \frac{1}{|B(r)|} \int_{B(r) \cap {}^c \Omega \cap E} \fw \d y.
\end{align}
If $B(r) \cap {}^c \Omega \cap E$ is a Lebesgue nullset for some $r > 0$, then $\fw(x) = 0$ holds and we can stop here. Otherwise, we use the identity 
\begin{align*}
 |A \cap C| - |A \cap B| + |A \cap B \cap {}^c C| = |A \cap {}^c B \cap C|
\end{align*}
for measurable sets $A,B,C \subseteq \R^n$ to write
\begin{align*}
 \frac{1}{|B(r)|} \int_{B(r) \cap {}^c \Omega \cap E} \fw \d y = \mu(r) \barint_{B(r) \cap {}^c \Omega \cap E} \fw \d y
\end{align*}
with $\mu$ given by
\begin{align*}
 \mu(r) = \frac{|B(r) \cap E|}{|B(r)|} - \frac{|B(r) \cap \Omega|}{|B(r)|} + \frac{|B(r) \cap \Omega \cap {}^c E|}{|B(r)|}.
\end{align*}
So, taking into account \eqref{eq3:iitoi} and the continuity of $\fw|_{E}$ at $x$ we finally arrive at 
\begin{align*}
 \fw(x) = \liminf_{r \to 0} \mu(r) \cdot \fw(x)
\end{align*}
and in order to deduce $\fw(x) = 0$ it remains to make sure that the limit inferior is different from $1$. Thanks to the maximal Lebesgue density of $E$ at $x$, see \eqref{eq1:iitoi}, and as $x$ is contained in the open set $U_j$, we can simplify
\begin{align}
\label{eq4:iitoi}
 \liminf_{r \to 0} \mu(r) = 1 - \limsup_{r \to 0}\frac{|B(r) \cap U_j \cap \Omega|}{|B(r)|} + 0. 
\end{align}
In order to handle the middle term, we recall the following fundamental property of Sobolev extension domains~\cite[Thm.~2]{Hajlasz-Koskela-Tuominen}.

\begin{lemma}
\label{lem:Hajlasz}
If a domain $V \subseteq \R^n$ admits a bounded Sobolev extension operator $\E: \W^{1,p}(V) \to \W^{1,p}(\R^n)$ for some $p \in [1,\infty)$, then $V$ is $n$-Ahlfors regular.
\end{lemma}

\noindent Owing to Assumption~\ref{ass:W1pExtension}, this lemma in particular applies to $V = U_j \cap \Omega$. Hence, there is a constant $c>0$ such that for all $y \in U_j \cap \Omega$ and all $r <1$ it holds
\begin{align*}
 \frac{|B(y,r) \cap U_j \cap \Omega|}{|B(y,r)|} \geq c.
\end{align*}
Since $x \in D \cap U_j$ lies on the boundary of $U_j \cap \Omega$, we can find for any $0<r<1$ a point $y \in  U_j \cap \Omega$ such that $B(y,r/2) \subseteq B(r) \subseteq B(y,2r)$. Thus,
\begin{align*}
 \frac{|B(r) \cap U_j \cap \Omega|}{|B(r)|} 
 \geq \frac{|B(y,r/2) \cap U_j \cap \Omega|}{|B(y,2r)|}
 \geq c \frac{|B(y,r/2)|}{|B(y,2r)|} = 4^{-n}c.
\end{align*}
In particular, going back to \eqref{eq4:iitoi} we obtain 
\begin{align*}
 \liminf_{r \to 0} \mu(r) \leq 1 - 4^{-n}c < 1
\end{align*}
and we had already convinced ourselves that this implies $\fw(x) = 0$. 

\begin{remark}
\label{rem:key}
In view of Lemma~\ref{lem:Hajlasz} we see that in the equation below \eqref{eq2:iitoi} our assumption on $u$ enabled us to  neglect the integral over a substantial part of $B(r)$, namely $B(r) \cap \Omega$. This is the key point in the proof. 
\end{remark}

\medskip

\emph{Step 4: Conclusion of the case $j \geq 1$}. So far we have shown that the precise representative $\fw$ of $w = |v_j|$ vanishes for $\gp$-almost every $x \in D$. In particular,
\begin{align*}
 \lim_{r \to 0} \bigg|\frac{1}{|B(x,r)|} \int_{B(x,r)} v_j \d y \bigg| = 0
\end{align*}
holds for $\gp$-almost every $x \in D$. Since $v_j$ is a Sobolev extension of $u_j$, the implication `(iii) $\Longrightarrow$ (i)' proved before yields $u_j \in \W_D^{1,p}(\Omega)$ as desired.

\subsection*{(i) \texorpdfstring{$\Longrightarrow$}{=>} (iii)}
This is precisely the statement of \cite[Thm.~6.9]{Hardy-Poincare} and the Havin-Bagby-Theorem. The proof relies on Assumption~\ref{ass:W1pExtension}, a localization procedure, and Lemma~\ref{lem:Hajlasz}.

\subsection*{(i) \texorpdfstring{$\Longrightarrow$}{=>} (ii)}
Let $u \in \W_D^{1,p}(\Omega)$. In view of the implication `(i) $\Longrightarrow$ (iii)' we have a Sobolev extension $v \in \W^{1,p}(\R^n)$ whose precise representative $\fv$ vanishes $\gp$-almost everywhere on $D$ at our disposal. The Lebesgue Differentiation Theorem for Sobolev functions discussed in Section~\ref{Sec:Preliminaries} yields
\begin{align*}
 \lim_{r \to 0} \frac{1}{|B(x,r)|} \int_{B(x,r)} |v| \d y = 0
\end{align*}
for $\gp$-almost every $x \in D$. Since $v$ extends $u$, we have
\begin{align*}
 \frac{1}{|B(x,r)|} \int_{B(x,r) \cap \Omega} |u| \d y \leq \frac{1}{|B(x,r)|} \int_{B(x,r)} |v| \d y 
\end{align*}
and the conclusion follows.

\subsection*{(iii) \texorpdfstring{$\Longleftrightarrow$}{<=>} (iv)}

The implication `(iii) $\Longrightarrow$ (iv)' is a direct consequence of the comparison principle stated in Lemma~\ref{lem:ComparisonPrinciple}. It does not require $D$ to be $l$-Ahlfors regular. As for the reverse implication, it has been shown in \cite[Cor.~4.5]{TripleMitrea-Brewster} that if $v \in \W^{1,p}(\R^n)$ is such that
\begin{align*} 
 \lim_{r \to 0} \frac{1}{|B(x,r)|} \int_{B(x,r)} v \d y = 0 
\end{align*}    
holds for $\cH_l$-almost every $x \in D$ and if $D$ is $l$-Ahlfors regular with parameter $l \in (n-p, n)$, then the same convergence already holds for $\gp$-almost every $x \in D$. In fact, this is a rather direct consequence of the deep extension/restriction-theory for Besov spaces on Ahlfors-regular sets developed by Jonsson and Wallin \cite{Jonsson-Wallin}.

\subsection*{(i) \texorpdfstring{$\Longleftrightarrow$}{=>} (v)}
This is precisely the main result on Hardy's inequality for Sobolev functions vanishing on a part of the boundary obtained in \cite[Thm.~3.2 \& 3.4]{Hardy-Poincare}.

\subsection*{(ii) \texorpdfstring{$\Longrightarrow$}{=>} (vi)}
This implication follows once again from the comparison principle.

\subsection*{(vi) \texorpdfstring{$\Longrightarrow$}{=>} (iii)}
This will be obtained by re-running the proof of `(ii) $\Longrightarrow$ (i)'. First, we split $u = \sum_{j = 0}^N u_j$ as before. Concerning $u_0$, our assumption (vi) and the support property of $\eta_0$ imply
\begin{align*}
 \lim_{r \to 0} \frac{1}{|B(x,r)|} \int_{B(x,r) \cap \Omega} |u_0| \d y = 0,
\end{align*}
now only for $\cH_l$-almost every $x \in \partial \Omega$. The second part of Proposition~\ref{Prop:SZ} yields that the zero extension $u_0^*$ is a Sobolev extension of $u_0 \in \W^{1,p}(\Omega)$ with the property required in (iv) and we deduce $u_0 \in \W^{1,p}_D(\Omega)$ from the equivalence with (i). 

Turning to $u_j$ in the case $j \geq 1$, the difference with the proof of `(ii) $\Longrightarrow$ (i)' is that the exceptional set designed in Step~2 is only of vanishing $\cH_l$-measure. However, then we can apply Step~3 \emph{verbatim} to obtain that the extension $v_j \in \W^{1,p}(\R^n)$ satisfies
\begin{align*}
 \lim_{r \to 0} \bigg|\frac{1}{|B(x,r)|} \int_{B(x,r)} v_j \d y \bigg| = 0
\end{align*}
for $\cH_l$-almost every $x \in D$. Hence, $v_j$ has again the property required in (iv) and we conclude $u_j \in \W^{1,p}_D(\Omega)$ as before. \hspace*{\fill} $\square$
\section{Counterexamples}
\label{Sec:Examples}

\noindent We provide two examples showing that without a certain regularity assumption on $\Omega$ near the common frontier of $D$ with its complementary boundary part (as guaranteed by Assumption~\ref{ass:W1pExtension}) the equivalence of (i) and (ii) in Theorem~\ref{thm:main} can fail in both directions. For simplicity of exposition both examples are constructed in the plane but the construction can easily be transferred to higher dimensions.

To begin with, we construct a fractal domain $\Omega \subseteq \R^2$ depending on two sequences of positive parameters $\{a_j\}_{j \in \IN}$ and $\{b_j\}_{j \in \IN}$. Here, $\IN = \{0,1,\ldots\}$. We consider the infinite graph consisting of all edges and vertices of the collection of dyadic squares 
\begin{align*}
\bigg\{ \bigg[\frac{k}{2^j}, \frac{k+1}{2^j}\bigg] \times \bigg[\frac{1}{2^j}, \frac{1}{2^{j-1}}\bigg] : j,k \in \IN,\, k \leq 2^j -1 \bigg\}
\end{align*}
displayed on the left of Figure~\ref{fig1}. For each $j \in \IN$ it contains exactly $2^j$ horizontal and $2^j +1$ vertical edges of length $2^{-j}$, which we denote from left to right by $h_j^0,\ldots, h_j^{2^j-1}$ and $v_j^0,\ldots, v_j^{2^j}$, respectively. From this `skeleton' we construct the domain $\Omega$ by blowing up the line segments $h_j^k$ and $v_j^k$ to open rectangles $H_j^k = h_j^k + (-a_j,a_j)^2$ and $V_j^k = v_j^k + (-b_j,b_j)^2$:
\begin{align*}
 \Omega := \bigcup_{j =0}^\infty \bigg( \bigcup_{k=0}^{2^j-1} H_j^k \cup \bigcup_{k=0}^{2^j} V_j^k \bigg),
\end{align*}
compare with Figure~\ref{fig1}. Here, we write $A + B = \{a+b : a \in A, \, b \in B\}$ for the sum of two sets $A, B \subseteq \R^2$, so that for example $H_j^k$ has horizontal side length $2^{-j} + 2a_j$ and vertical side length $2a_j$. We shall always choose $0<a_j,b_j <2^{-j-1}$ in order to arrange the overlap of the horizontal and vertical rectangles as displayed schematically in Figure~\ref{fig1}. Note that the Dirichlet part $D := [0,1] \times \{0\}$ is a closed, $1$-Ahlfors regular subset of $\partial \Omega$ and that $\Omega$ exhibits Lipschitz coordinate charts around every boundary point $x \in \partial \Omega \setminus D$. 

\begin{figure}[ht]
\begin{center}
 \includegraphics[scale=.6]{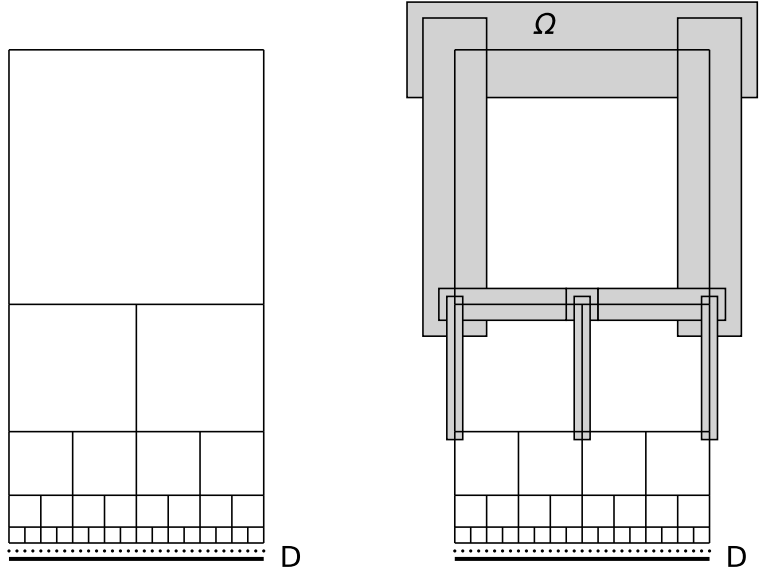}
 \caption{The dyadic `skeleton' of $\Omega$ is obtained from the square $[0,1] \times [1,2]$ by iteratively attaching a total number of $2^j$ disjoint squares of side length $2^{-j}$ at the bottom of the existing construction. The domain $\Omega$ is then constructed by blowing up the line segments to appropriately sized open rectangles.}
 \label{fig1}
\end{center}
\end{figure}

\begin{example}
\label{Ex:W1pD but no averages}
We let $1<p < \infty$ and construct $\Omega$ using the sequences $a_j = 2^{-j-2}$ and $b_j = 2^{-(1+p)j}$. We claim that the constant function $u = 1$ is contained in $u \in \W_D^{1,p}(\Omega)$ although the condition in part (ii) of Theorem~\ref{thm:main} fails at every boundary point $x \in D$. 

To see the second claim, let $x \in D$ and $0 < r < 1$. If $j \in \IN$ satisfies $2^{-j+1} \leq r/2$, then $B(x,r) \cap \Omega$ contains a rectangle of horizontal side length $r/2$ and vertical side length $a_j = 2^{-j-2}$. Thus,
\begin{align*}
 \frac{1}{|B(x,r)|} \int_{B(x,r) \cap \Omega} |u| \d y \geq \frac{1}{\pi r^2} \sum_{\substack{j \in \IN \\ 2^{-j+1} \leq r/2}} r 2^{-j-3} \geq \frac{1}{32 \pi},
\end{align*}
showing that the condition in part (ii) of Theorem~\ref{thm:main} fails. On the other hand,
\begin{align*}
 u_j: \Omega \to [0,1], \quad u_j(y_1,y_2) = \begin{cases}
                                              0 &\text{if $y_2 < 2^{-j-1} + 2^{-j-3}$}   \\                                           
                                              1 &\text{if $y_2 > 2^{-j} - 2^{-j-2}$}\\
                                              2^{j+3}y_2 - 5 &\text{else}
                                             \end{cases}
\end{align*}
is continuous, piecewise affine, and its support is disjoint from $D$. Lebesgue's theorem guarantees $u_j \to u$ in $\L^p(\Omega)$ and since by construction $\nabla u_j$ is supported in the set $\bigcup_{k=0}^{2^j} V_j^k$ and satisfies the pointwise bound $|\nabla u_j| \leq 2^{j+3}$, we also obtain
\begin{align*}
 \int_{\Omega} |\nabla u_j|^p \d y 
 &\leq \sum_{k=0}^{2^j} 2^{(j+3)p} |V_j^k|
 \lesssim 2^{-j},
\end{align*}
that is, $\nabla u_j \to \nabla u$ in $\L^p(\Omega)$. In order to conclude $u \in \W^{1,p}_D(\Omega)$ it suffices to convolve the approximants $u_j$ by smooth kernels with sufficiently small support.
\end{example}

\begin{example}
\label{Ex:Averages but no W1pD}
We let $4 < p < \infty$ and construct $\Omega$ using $a_j = b_j = 4^{-j-1}$. We claim that this time the constant function $u = 1$ is not contained in $\W_D^{1,p}(\Omega)$ although the condition in part (ii) of Theorem~\ref{thm:main} holds at every boundary point $x \in D$.

In order to see the second claim, let $x \in D$ and $0 < r < 1$. For each $j \in \IN$ it follows from the dyadic structure of the skeleton for $\Omega$ that $B(x,r)$ intersects at most $\lfloor2r/2^{-j} + 3 \rfloor$ of the vertical rectangles $V_j^k$, each of which has measure $|V_j^k| \leq 2^{-3j}$. As for the horizontal rectangles, we simply observe that $\bigcup_{k=0}^{2^j-1} H_j^k \cap B(x,r)$ is contained in a rectangle with side lengths $2r$ and $2^{-2j-1}$. In conclusion,
\begin{align*}
 \frac{1}{|B(x,r)|} \int_{B(x,r)} |u| \d y 
 \leq \frac{1}{\pi r^2} \sum_{\substack{j \in \IN \\ 2^{-j} \leq r}} \lfloor r 2^{j+1} + 3 \rfloor 2^{-3j} + r 2^{-2j}
 \lesssim r,
\end{align*}
taking care of the condition in part (ii) of Theorem~\ref{thm:main}. 

Next, we shall prove that despite its rather irregular structure the domain $\Omega$ still admits the Poincar\'{e} inequality
\begin{align}
\label{eq:Ex2 Poincare}
 \|v\|_{\L^\infty(\Omega)} \lesssim \|\nabla v\|_{\L^p(\Omega)} \qquad (v \in \W^{1,p}_D(\Omega)).
\end{align}
In particular, this implies $u \notin \W^{1,p}_D(\Omega)$. By density we can assume $v \in \C_D^\infty(\Omega)$. Since $p>2 = n$, there is a constant $C>0$ depending only on $p$ such that on every open square $Q \subseteq \R^2$ with sidelength $\ell(Q)>0$ we have Morrey's estimate
\begin{align}
\label{eq:Ex2 Morrey}
 |v(a) - v(b)| \leq C \ell(Q)^{1-2/p}\|\nabla v\|_{\L^p(Q)} \qquad (a,\,b \in Q),
\end{align}
see for instance \cite[Lem.~7.12 \& 7.16]{Gilbarg-Trudinger}. Next, we consider a rectangle $R_j \subseteq \Omega$ of side lengths $2 \cdot 4^{-j-1}$ and $2^{-j} + 2 \cdot 4^{-j-1}$ for some $j \in \IN$, for example one of the $V_j^k$ or $H_j^k$. Any two points $a, b \in R_j$ can be joined by a chain of squares $Q_1,Q_2,\ldots,Q_{2^{j+4}}$ with radii $4^{-j-1}$ that are all contained in $R_j$ and have the properties $a \in Q_1$, $b \in Q_{2^{j+4}}$, and $Q_m \cap Q_{m+1} \neq \emptyset$ for $1 \leq m \leq 2^{j+4}-1$. By a telescoping sum and \eqref{eq:Ex2 Morrey},
\begin{align}
\label{eq:Ex2 Morrey rec}
 |v(a) - v(b)| \leq \sum_{m=1}^{2^{j+4}} C 4^{(j+1)(2/p - 1)} \|\nabla v\|_{\L^p(Q_m)} \leq 16C 2^{j(4/p -1)} \|\nabla v\|_{\L^p(\Omega)}.
\end{align}
Finally, let $y \in \Omega$. There exist $j' \in \IN$ and $0 \leq k \leq 2^{j'}$ such that $y \in V_{j'}^k$ or $y \in H_{j'}^k$. In the first case we consider the chain of rectangles $R_j := V_j^k$, $j \geq j'$, which have the property that $y \in R_{j'}$ and $R_j \cap R_{j+1} \neq \emptyset$ for all $j \geq j'$. In the second case we add $R_{j'-1} := H_{j'}^k$ to the chain. Now, $v \in \C_D^\infty(\Omega)$ implies that $v = 0$ holds everywhere on $R_j$ for $j$ sufficiently large. Hence, \eqref{eq:Ex2 Morrey rec} and another telescoping sum yield
\begin{align*}
 |v(y)| \leq 32C \|\nabla v\|_{\L^p(\Omega)} \sum_{j =0}^\infty 2^{j(4/p -1)}
\end{align*}
and the geometric series converges due to our assumption $p>4$. This proves \eqref{eq:Ex2 Poincare}.
\end{example}


\begin{thebibliography}{99}
\bibitem{Adams-Hedberg}
\newblock \textsc{D.~R. Adams} and \textsc{L.~I. Hedberg}.
\newblock Function {S}paces and {P}otential {T}heory.
  Grundlehren der mathematischen Wissenschaften,  vol.~314,
\newblock Springer, Berlin, 1996.

\bibitem{TripleMitrea-Brewster}
\newblock \textsc{K.~Brewster}, \textsc{D.~Mitrea}, \textsc{I.~Mitrea}, and \textsc{M.~Mitrea}.
\newblock {\em Extending {S}obolev functions with partially vanishing traces
  from locally {$(\varepsilon,\delta)$}-domains and applications to mixed
  boundary problems\/}.
\newblock J. Funct. Anal. \textbf{266} (2014), no.~7, 4314--4421.

\bibitem{Hardy-Poincare}
\newblock \textsc{M.~Egert}, \textsc{R.~Haller-Dintelmann}, and \textsc{J.~Rehberg}.
\newblock {\em Hardy's inequality for functions vanishing on a part of the
  boundary\/}.
\newblock Potential Anal. \textbf{43} (2015), 49--78.

\bibitem{Rehberg-terElstADE}
\newblock \textsc{A.~F.~M.~ter Elst} and \textsc{J.~Rehberg}.
\newblock{\em H\"older estimates for second-order operators on domains with rough boundary\/}.
\newblock Adv. Differential Equations \textbf{20} (2015), no.~3-4, 299--360.

\bibitem{EG}
\newblock \textsc{L.~C. Evans} and \textsc{R.~F. Gariepy}.
\newblock Measure {T}heory and {F}ine {P}roperties of {F}unctions. Studies in Advanced Mathematics,
\newblock CRC Press, Boca Raton FL, 1992.

\bibitem{Federer}
\newblock \textsc{H.~Federer}.
\newblock Geometric {M}easure {T}heory. Die Grundlehren der mathematischen Wissenschaften, vol.~153,
\newblock Springer, New York, 1969.

\bibitem{currents}
\newblock \textsc{M.~Giaquinta}, \textsc{G.~Modica}, and \textsc{J.~Sou{\v{c}}ek}.
\newblock Cartesian currents in the calculus of variations {I}.
  Results in Mathematics and Related Areas. 3rd Series, vol.~37,
\newblock Springer-Verlag, Berlin, 1998.

\bibitem{Gilbarg-Trudinger}
\newblock \textsc{D.~Gilbarg} and \textsc{N.~S. Trudinger}.
\newblock Elliptic {P}artial {D}ifferential {E}quations of {S}econd {O}rder. Classics in Mathematics,
\newblock Springer, Berlin, 2001.

\bibitem{Hajlasz-Koskela-Tuominen}
\newblock \textsc{P.~Haj{\l}asz}, \textsc{P.~Koskela}, and \textsc{H.~Tuominen}.
\newblock {\em Sobolev embeddings, extensions and measure density condition\/}.
\newblock J. Funct. Anal. \textbf{254} (2008), no.~5, 1217--1234.

\bibitem{RobertJDE}
\newblock \textsc{R.~Haller-Dintelmann} and \textsc{J.~Rehberg}.
\newblock{\em Maximal parabolic regularity for divergence operators including mixed boundary conditions\/}.
\newblock J. Differential Equations \textbf{247} (2009), no.~5, 1354--1396.

\bibitem{Jonsson-Wallin}
\newblock \textsc{A.~Jonsson} and \textsc{H.~Wallin}.
\newblock {\em Function spaces on subsets of {${\R}^n$}\/}.
\newblock Math. Rep. \textbf{2} (1984), no.~1.

\bibitem{Swanson-Ziemer}
\newblock \textsc{D.~Swanson} and \textsc{W.~P.~Ziemer}.
\newblock {\em Sobolev functions whose inner trace at the boundary is zero\/}.
\newblock Ark. Mat. \textbf{37} (1999), no.~2, 373--380.

\bibitem{Yeh}
\newblock \textsc{J.~Yeh}.
\newblock Real analysis.
\newblock World Scientific Publishing, Hackensack NJ, 2006.

\bibitem{Ziemer}
\newblock \textsc{W.~P. Ziemer}.
\newblock Weakly differentiable functions. Graduate Texts in Mathematics, vol.~120,
\newblock Springer, New York, 1989.

\end{thebibliography}
\end{document}